\documentclass[11pt,final]{amsart}

\usepackage{amssymb}

\usepackage{graphicx}

\usepackage[color]{showkeys}
\usepackage{hyperref}
\definecolor{refkey}{gray}{.75}

\usepackage{booktabs}

\usepackage{enumerate}

\usepackage{xcolor}
\definecolor{rev}{RGB}{0,0,0}

\newtheorem{theorem}{Theorem}[section]
\newtheorem{lemma}[theorem]{Lemma}

\theoremstyle{definition}

\theoremstyle{remark}
\newtheorem{remark}[theorem]{Remark}

\numberwithin{equation}{section}

\def\p{\partial}
\def\tilde{\widetilde}

\def\bs{\boldsymbol}

\def\ub{\boldsymbol{u}}

\def\vb{\bs{v}}
\def\fb{\bs{f}}
\def\kb{\mathbf{k}}
\def\nb{\mathbf{n}}
\def\tb{\mathbf{t}}

\def\CE{\mathrm{CE}}
\def\VE{\mathrm{VE}}

\def\H{\mathcal{H}}
\def\D{\mathcal{D}}
\def\C{\mathcal{C}}
\def\R{\mathcal{R}}
\def\P{\mathcal{P}}

\newcommand{\ra}[1]{\renewcommand{\arraystretch}{#1}}

\begin{document}

\title[Staggered finite difference finite volume schemes]{Error
  analysis of staggered finite difference finite volume 
  schemes on unstructured 
  meshes}


\author[Q.~Chen]{Qingshan Chen}
\address{Department of Mathematical Sciences, Clemson University,
  Clemson, SC 29631, USA.}
\email{qsc@clemson.edu}
\thanks{This work is in part supported by a grant from the Simons
  Foundation (\#319070 to Qingshan Chen).  
}

\date{\today}


\begin{abstract}
This work combines the consistency in lower-order differential
operators with external approximations of functional spaces to obtain
error estimates for finite difference finite volume schemes on
unstructured non-uniform meshes. This combined approach is first
applied to the one-dimensional elliptic boundary value problem on
non-uniform meshes, and a first-order convergence rate is obtained,
which agrees with the results previously reported. The approach is
also applied to the staggered MAC scheme for the two-dimensional
incompressible Stokes problem on unstructured meshes. A
first-order convergence rate is obtained, which improves over
a previously reported result in that it also holds on unstructured
meshes. For both problems considered in this work, the convergence
rate is one order higher on meshes satisfying special requirements. 
\end{abstract}

\maketitle

\section{Introduction}
Staggered finite difference finite volume schemes (FDFV) place scalar
variables and vectorial variables at different grid points in a
staggered pattern. They inherit all the desirable conservative
properties from the finite volume method. The staggered placement of
the scalar and vectorial variables also render the schemes
particularly suitable for transport phenomena. A classical staggered
scheme is the famous Marker-and-Cell (MAC) scheme proposed by Harlow
and Welch (\cite{Harlow:1965jv}). The MAC scheme is considered the
method of choice for incompressible flows (see a comprehensive
treatment of this topic by Wesseling (\cite{Wesseling:2001ci}). Since
its introduction, it has also been argued that the MAC scheme is
suitable for flows of all speeds
(\cite{Harlow:1968ff,{Harlow:1971iv}}). A variant of the MAC scheme,
the so-called C-grid scheme, has enjoyed wide popularity in the ocean
and atmosphere modeling community for its superior performance in
resolving the inertial-gravity waves
(\cite{Arakawa:1977wr,{Thuburn:2009tb},{Ringler:2010io}}). See 
  \cite{Chen:2013bl,{Chen:2016ip}} for additional variants of the
  C-grid scheme. 

Compared to the finite element method, the analysis of the finite
volume method is less abundant, likely due to the lack of a natural
variational framework, \textcolor{rev}{as finite volume schemes are
  typically not written in the variational form.}
 However, due to the popularity of the finite
volume methods in engineering and geophysical applications,
theoretical analysis of these methods to determine their stability,
convergence and accuracy is much needed. 
Through a mapping from the trial function space to
the test function space, Li et al (\cite{Li:2000uz}) cast
non-staggered generalized finite difference schemes into a Galerkin
variational framework, where the error analysis can be performed in
the same way as
for the finite element methods (\cite{Chou:2001gi}). 
 Temam and coauthors employ functional
analytical tools to study the convergence of non-staggered finite
volume schemes on structured and quasi-structured meshes
(\cite{Faure:2006ky,{Adamy:2006hi}, {Gie:2010vy}, {Gie:2015tw}}),
\textcolor{rev}{by constructing external 
approximations of function 
spaces (\cite{Cea:1964vy}) and transforming the schemes into
the variational forms.}
 Chen
(\cite{Chen:2016uw}) extends this approach to staggered finite
difference finite volume schemes on  unstructured
meshes. Relying heavily on functional analytical tools, but without
explicit approximations of functional spaces, Eymard et al
(\cite{Eymard:2000tt}) study the convergence and accuracy for
non-staggered finite volume methods on structured and
unstructured meshes. \textcolor{rev}{Droniou et al
  (\cite{Droniou2013-mq}) establishes 
the convergence of a class of gradient schemes written in the discrete
variational formulations.  }

The current work is concerned with the error analysis of staggered
finite volume schemes on unstructured meshes. As
highlighted by Eymard et al (\cite{Eymard:2000tt}), a major hurdle to
deriving error estimates for finite volume scheme is the fact that
finite volume schemes are often not consistent with the differential
equations they are approximating. The authors show that, actually,
consistency in approximations of the fluxes is enough for error
estimates in many cases. The authors apply this idea to derive error
estimates for non-staggered finite volume schemes for linear elliptic
and parabolic problems. For scalar linear hyperbolic problems, the
finite volume schemes are consistent with the continuous equation, and
their error estimates can be handled in the same way as the finite
difference schemes. Error estimates for staggered finite volume
schemes are scarce. Nicolaides (\cite{Nicolaides:1992vs}) proves that
the MAC scheme for the incompressible Stokes problem on a rectangular
mesh is first-order accurate in both $L^2$ and $H^1$ norm. 

The main thrust of this work is to show that, in the spirit of
\cite{Eymard:2000tt}, consistency in lower-order differential
operators, such as derivative or curl, is sufficient for error
estimates in many cases. Here, we combine the consistency in
lower-order operators with the external approximation framework to
derive error estimates for staggered FDFV schemes. 
To illustrate how this combined approach works, we first apply
it to a simple one-dimensional elliptic problem, and recover the
results already presented in \cite{Eymard:2000tt}. Then, in Section 3,
we apply the combined approach to the MAC scheme for the Stokes
problem on unstructured meshes. We show that the scheme is first-order
accurate in both $L^2$ and $H^1$ norms, and in certain special cases,
including the rectangular meshes, the scheme is second-order accurate
under both $L^2$ and $H^1$ norms.

\section{The one-dimensional elliptic
  problem}\label{sec:one-dimens-ellipt}
We consider the one-dimensional elliptic problem
\begin{align}
  &-u_{xx} = f,\qquad 0< x < 1,\label{eq:1}\\
  & u(0 ) = u(1) = 0.\label{eq:2}
\end{align}
We define the function spaces
\begin{align*}
  V &= H_0^1(0,1),\\
  V' &= H^{-1}(0,1).
\end{align*}
For $u,\,v \in V$, we define the bilinear operator
\begin{equation*}
  a(u,\,v) = (u_x,\, v_x).
\end{equation*}
Here, $(\cdot,\,\cdot)$ denotes the inner product of the $L^2(0,1)$
space.
The weak formulation of the elliptic problem
\textcolor{rev}{\eqref{eq:1}--\emph{\eqref{eq:2}}} can now be stated:
\begin{quote}
  For each $f\in V'$, find $u\in V$ such that
  \begin{equation}
    \label{eq:3}
    a(u,\,v) = \langle f,\,v\rangle,\qquad \forall\,\, v\in V.
  \end{equation}
\end{quote}
If $u$ is a classical solution of the elliptic problem
\eqref{eq:1}--\eqref{eq:2}, then $u$ 
satisfies \eqref{eq:3} as well. In this sense, \eqref{eq:3} is a weak
formulation of \eqref{eq:1}--\eqref{eq:2}. The existence and
uniqueness of 
a solution to the weak problem is now part of the classical elliptic
PDE theory, which is based on the fact that the
bilinear operator $a(\cdot,\cdot)$ is coercive. The classical theory
also provides that, if the right-hand side forcing $f$ is smooth, then
the solution is also smooth. 

\subsection{External approximation of $V$ and the finite volume
  scheme}\label{sec:extern-appr-v} 
\begin{figure}[h]
  \centering
  \includegraphics[width=5.5in]{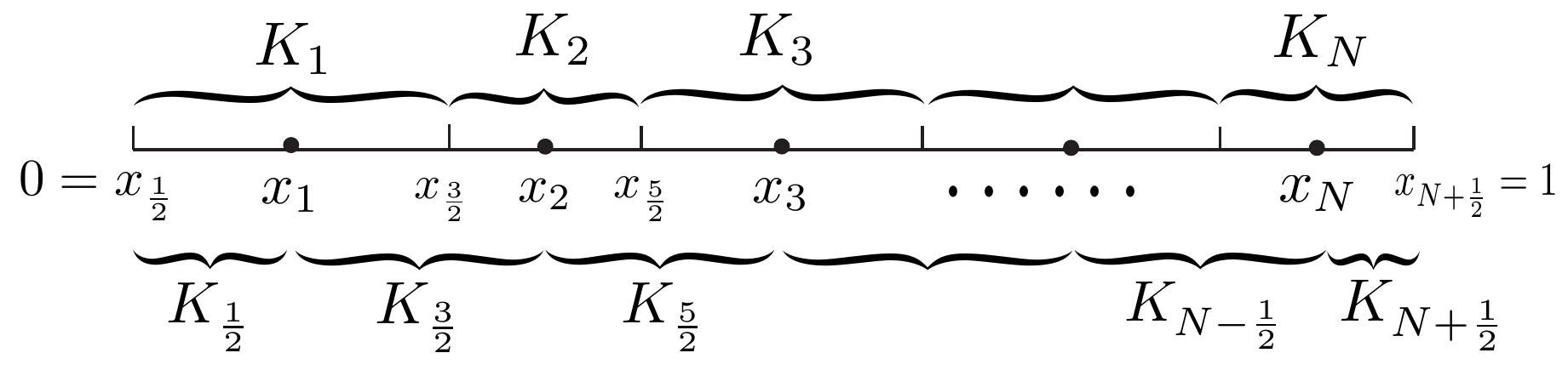}
  \caption{One-dimensional non-uniform mesh}
  \label{fig:1d-mesh}
\end{figure}

The interval $[0,\,1]$ is
partitioned into $N$ subintervals $K_i$, $1\leq i\leq
N$, which are centered at $x_i$, and have lengths of $h_i$ (Figure
\ref{fig:1d-mesh}). The duals of $K_i$ are 
$K_{i+\frac{1}{2}}$, $0\leq i\leq N$, with the corresponding centers
$x_{i+\frac{1}{2}}$, and lengths $h_{i+\frac{1}{2}}$. Here the term
``center'' is used in its liberal sense; $x_i$
(resp.~$x_{i+\frac{1}{2}}$) is not necessarily the midpoint of $K_i$
(resp.~$K_{i+\frac{1}{2}}$).  The formal notations
about the mesh are given below:
\begin{align*}
  &x_i \in K_i = [x_{i-\frac{1}{2}},\, x_{i+\frac{1}{2}}],\qquad  1\leq i \leq
  N\\ 
  &x_{i+\frac{1}{2}} \in K_{i+\frac{1}{2}} = [x_i,\,x_{i+1}]\textcolor{rev}{,} \qquad
  1\leq i\leq N-1\\ 
  &x_\frac{1}{2} \in K_{\frac{1}{2} } = [x_{\frac{1}{2}},\, x_1],\quad
  x_{N+\frac{1}{2}}\in K_{N+\frac{1}{2}}
  = [x_N,\,x_{N+\frac{1}{2}}],\\
  &h_i = |K_i|,\qquad 1\leq i\leq N,\\
  &h_{i+\frac{1}{2}} = |K_{i+\frac{1}{2}}|,\qquad 0\leq i\leq N.
\end{align*}
\textcolor{rev}{We are interested in non-uniform meshes, and thus the
  mesh sizes $h_i$
and $h_{i+\frac{1}{2}}$ are non-constant. However, for regularity
  reasons, we assume that the ratio between the largest mesh size and
  the smallest mesh size is bounded. More precisely, we assume that,
  for each mesh, there exists $h>0$ such that
  \begin{equation*}
    m h < h_i, h_{i+\frac{1}{2}} < Mh,
  \end{equation*}
where $m$ and $M$ are constants that are independent of the meshes. In
what follows, each mesh is represented by a subscript ${}_h$, and the
total set of admissible meshes is designated as $\mathcal{H}$.}

We let
\begin{equation*}
  F = L^2(0,\,1)\times L^2(0,\,1). 
\end{equation*}
We define an isomorphism $\Pi$ from $V$ into $F$ as
\begin{equation}
  \label{eq:4}
  \Pi u = (u,\,u_x)\in F,\qquad \forall\, u \in V,
\end{equation}
\textcolor{rev}{where $u_x$ designates the first derivative of the
  function.} 
For each $h\in\mathcal{H}$,
the space $V$ is approximated by $V_h$ defined on the primary mesh and
$\tilde V_h$ defined on the dual mesh,
\begin{align*}
  V_h =& \left\{ u_h = \sum_{i=1}^N u_i
    \chi_i\right\}\textcolor{rev}{,}\\ 
  \tilde V_h =& \left\{ \tilde u_h = \sum_{i=0}^N  \tilde
      u_{i+\frac{1}{2}} \chi_{i+\frac{1}{2}} \right\}.\\
\end{align*}
Here, $\chi_i$ and $\chi_{i+\frac{1}{2}}$ are the characteristic
functions defined on the mesh cells $K_i$ and $K_{i+\frac{1}{2}}$,
respectively. 
We define the restriction operator $\R_h$ from $V$ into $V_h$,
\begin{equation}
\label{eq:70}
  \R_h u = \sum_{i=1}^N u(x_i) \chi_i\in V_h,\qquad \forall u\in V. 
\end{equation}
We note that this definition makes sense because $H^1_0 (0,\,1)
\subset \C[0,\,1]$. 

\begin{remark}
\textcolor{rev}{Some authors distinguish between finite difference
  (FD) and finite 
volume (FV)
schemes by whether the discrete variables are defined as the pointwise
values of 
the corresponding continuous variables (FD), or as  the area/volume
averaged values of the corresponding
continuous variables (FV); see \cite{Faure:2006ky}. 
However, taking the discrete variable as the area/volume average of
the continuous variable can lead to serious consistency issues,
especially on non-uniform grids, as pointed out by \cite{Eymard:2000tt}. 
We follow the latter reference and distinguish FD and FV schemes by
how the differential operators are 
discretized. If the differential operators are discretized by Green's
or the divergence theorems, then the scheme is FV; if instead, all the
individual derivatives are discretized by Taylor series expansion,
then the scheme is FD. In fact, in most thus-defined FV schemes, the
FD formulae are also used on certain parts of the equations, or on the
fluxes that are part of the FV discretizations of  the differential
operators. For this reason, we sometimes call 
these schemes mixed finite difference finite volume schemes (FDFV,
e.g.~\cite{Ringler:2010io, Chen:2013bl}).  }
\end{remark}

\begin{remark}\textcolor{rev}{
  Our scheme, to be given later, will be classified as a finite
  difference scheme according to the definition adopted by
  \cite{Faure:2006ky}, because the discrete variable is viewed as the
  pointwise value of the continuous variable \eqref{eq:70}. However,
  the 
  one-dimensional grid that is used in this section is identical to
  the cell-centered FV grid, and differs from the FD grid in the same
  reference.  Our own definition regrading FD and FV schemes, given in the
   preceding remark, does not make
   a difference on the problem under consideration in the
   one-dimensional space.
  Due to the technical differences with \cite{Faure:2006ky} in the
  meshes specifications and for the sake of 
  completion, we provide full details on 
  specifications and 
  convergence analysis of the scheme below, while we acknowledge that these
  developments mostly parallel those in the cited reference, and
  differ in technical details.  
  The error analysis part (Section \ref{sec:error-estimates}) is new.}
\end{remark}

We define the gradient operators $\nabla_h$ for
$V_h$,
\begin{equation}
  \nabla_h u_h = \sum_{\textcolor{rev}{i=0}}^N \dfrac{u_{i+1}- u_i}{h_{i+\frac{1}{2}}}
  \chi_{i+\frac{1}{2}} \in \tilde V_h,\qquad\forall u\in V.\label{eq:5}
\end{equation}
The gradient operator is well-defined throughout the whole interval
under the convention that 
$$u_0 = u_{N+1} =0,$$ 
which agrees with the
homogeneous Dirichlet boundary conditions of $V$. 
The gradient operator $\tilde\nabla_h$ on the dual function space
$\tilde V_h$ is defined as
\begin{equation}
  \tilde\nabla_h \tilde u_h = \sum_{i=1}^N \dfrac{\tilde
    u_{i+\frac{1}{2}} - \tilde u_{i-\frac{1}{2}}}{h_i} \chi_i\in
    V_h,\qquad \forall\, \tilde u_h \in V_h.\label{eq:6}
\end{equation}
The space $V_h$ is equipped with the semi-$H^1$ norm,
\begin{equation*}
  \| u_h \| = |\nabla_hu_h|_0.
\end{equation*}
In fact, due to the Poincar\'e inequality that we will establish
later, $V_h$ is a Hilbert space under this norm.

Finally, we define the prolongation operator $\P_h$ from $V_h$ into
$F$,
\begin{equation}
  \label{eq:7}
  \P_h u_h = (u_h, \nabla_h u_h)\in F,\qquad \forall\, u_h \in V_h. 
\end{equation}

For each $h\in\H$, the external approximation of the function space
$V$ consists of  the function space $F$, the isomorphic mapping $\Pi$,
and the triplet $(V_h,\,\R_h,\,\P_h)$. 
Based on the specification of the norms of $V_h$ and $F$, it is
straightforward to verify the following result.
\begin{lemma}\label{lem:stability}
  The prolongation operator $\P_h$ from $V_h$ into $F$ is stable.
\end{lemma}

In subsequent analyses, we will need the following discrete
integration by parts formula,
\begin{lemma}
  Let $u_h\in V_h$, $v_h\in \tilde V_h$. Then
  \begin{equation}
    \label{eq:15}
    (\nabla_h u_h,\, v_h) = -(u_h,\,\tilde\nabla_h v_h).
  \end{equation}
\end{lemma}
\begin{proof} By definition,
  \begin{align*}
    (\nabla_h u_h,\, v_h) &= \sum_{i=0}^N
    \dfrac{u_{i+1}-u_i}{h_{i+\frac{1}{2}}} v_{i+\frac{1}{2}}
    h_{i+\frac{1}{2}}\\ 
    &= \sum_{i=0}^N (u_{i+1} - u_i)
    v_{i+\frac{1}{2}}\textcolor{rev}{.}
  \end{align*}
Using the conventions on the boundary terms $u_0$ and $u_{N+1}$, we
can rewrite the summation as 
  \begin{align*}
    (\nabla_h u_h,\, v_h) &= \sum_{i=1}^N  u_i
    (v_{i-\frac{1}{2}} - v_{i+\frac{1}{2}})\\
    & = -\sum_{i=1}^N u_i \dfrac{v_{i+\frac{1}{2}} -
      v_{i-\frac{1}{2}}}{h_i} h_i\\
    &= -(u_h,\,\tilde\nabla_h v_h).
  \end{align*}
\end{proof}

The convergence of external approximations to a function space has two
conditions. The first states that, for any $u\in V$, $\P_h\R_h u$
should converge to $\Pi u$ under the strong topology of $F$. This
condition ensures the consistency of the restriction and the
projection operators. The other condition requires that if $\P_h u_h$
convergences weakly to $\omega$ in $F$, then $\omega$ is the image of
an element of $V$. For the external approximation we have just
constructed, the following lemma confirms that both of these convergence
conditions are true.  
\begin{lemma}\label{lem:convergences}
 For the external approximation, the following hold
  true.\\ 
  \begin{enumerate}[({C}1)]
  \item For arbitrary $u\in V$,
    \begin{equation}
      \label{eq:8}
      \P_h \R_h u \longrightarrow \Pi u\qquad\textrm{strongly in } F.
    \end{equation}
  \item Let $\{u_h\}$ be a sequence in $F$. If
    \begin{equation*}
      \P_h u_h\rightharpoonup w\qquad\textrm{weakly in } F,
    \end{equation*}
    then, for some $u\in V$,
    \begin{equation*}
      w = \Pi u.
    \end{equation*}
  \end{enumerate}
\end{lemma}
\begin{proof}
  For $(C1)$,
  \begin{align*}
    \R_h u &= \sum_{i=1}^N u(x_i) \chi_i,\\
    \P_h \R_h u &= \left(\R_h u,\,\nabla_h\R_h u\right)\\
    &= \left(\sum_{i=1}^N u(x_i) \chi_i, \,\sum_{i=0}^N
      \dfrac{u(x_{i+1}) - u_{x_i}}{h_{i+\frac{1}{2}}}
      \chi_{i+\frac{1}{2}}\right).
  \end{align*}
As an easy application of the Taylor series expansion, it can be shown
that, if $u\in\D(0,1)$, then 
\begin{align*}
  \R_h u& \longrightarrow u\textcolor{rev}{,} & &\textrm{strongly in } L^2(0,1),\\
  \nabla_h \R_h u& \longrightarrow u_x\textcolor{rev}{,} &
  &\textrm{strongly in }  L^2(0,1)\textcolor{rev}{.}\\ 
\end{align*}
Then the definition of $\R_h$ can be extended to the whole space $V$
and the above convergence still holds. 

For $(C2)$, the weak convergence 
\begin{equation*}
  \P_h u_h \rightharpoonup (u,w)\textrm{ in } F
\end{equation*}
implies that
\begin{align}
  u_h & \rightharpoonup u\textrm{ in } L^2(0,1),\label{eq:10}\\
  \nabla_h u_h & \rightharpoonup w\textrm{ in }
  L^2(0,1).\label{eq:11} 
\end{align}
Let $v\in\D(0,1)$, and $\tilde\R_h v = \sum_{i=1}^N
v(x_{i+\frac{1}{2}})\chi_{i+\frac{1}{2}}$. It can be shown as in part
(a) that 
 \begin{equation}
   \left(\tilde\R_h v,\, \tilde\nabla_h\tilde\R_h v\right)\longrightarrow
     (v,\,v_x)\qquad \textrm{strongly in } F. \label{eq:12}
 \end{equation}
By the integration by parts formula, we have
\begin{equation}
  \label{eq:9}
  \left(\nabla_h u_h,\,\tilde\R_h v\right) = -\left( u_h,\,
    \tilde\nabla_h\tilde\R_h v\right).
\end{equation}
By the weak convergences in \eqref{eq:10}-~\eqref{eq:11} and the
strong convergences in \eqref{eq:12}, we obtain that
\begin{equation}
  \label{eq:13}
  (w,\,v) = -(u,\,v_x),\qquad\textrm{for any } v\in \D(0,1).
\end{equation}
This shows that 
\begin{equation}
  \label{eq:14}
  \dfrac{\p u}{\p x} = w\qquad\textrm{ in } \D'(0,1).
\end{equation}
Let $u^\#$ be the extension of $u$ by zero outside $(0,1)$, and let
$w^\#$ be the extension of $w$ by zero outside $(0,1)$. It can be
shown by the same procedure that
\begin{equation*}
  \dfrac{\p u^\#}{\p x} = w^\# \qquad\textrm{ in } \D'(0,1).
\end{equation*}
Thus $u^\#\in H^1(\mathbb{R}).$
The fact that $u^\#=0$ outside the interval $(0,1)$ indicates that
$u^\#$ vanishes on the boundary. Thus,
\begin{equation*}
  u = u^\#\bigr|_{(0,1)} \in H^1_0(0,1),
\end{equation*}
and the claim is proven.
\end{proof}

We can now state the finite volume scheme for the 1D elliptic boundary
value problem \eqref{eq:1}--\eqref{eq:2}:
\begin{quote}
  For each $f\in L^2(0,1)$, let $f_i = \frac{1}{h_i}\int_{K_i} fdx$
  and $f_h = \sum_{i=1}^N f_i\chi_i$. Find $u\in V_h$ such that
  \begin{equation}
    \label{eq:16}
    -\tilde\nabla_h\nabla_h u_h = f_h.
  \end{equation}
\end{quote}
By the integration by parts formula \eqref{eq:15}, we obtain the
variational form of the scheme,
\begin{equation}
  \label{eq:17}
  \left(\nabla_h u_h,\, \nabla_h v_h\right) = \left(f,\,
    v_h\right),\qquad \forall v_h\in V_h.
\end{equation}

\subsection{Existence and uniqueness of a
  solution}\label{sec:exist-uniq-solut}
We first establish the discrete Poincar\'e inequality.
\begin{lemma}
  If $u_h\in V_h$, then
  \begin{equation}
    \label{eq:18}
    |u_h|_0 \leq C|\nabla_h u_h|_0
  \end{equation}
for some $C>0$ independent of $h$.
\end{lemma}
\begin{proof} The procedure mimics that for the continuous case. 
  We let
  \begin{equation*}
    u_h = \sum_{i=1}^N u_i \chi_i.
  \end{equation*}
Then, by definition, we have
\begin{equation*}
  |u_h|_0^2 = \sum_{i=1}^N u_i^2 h_i.
\end{equation*}
For each individual $i$, $u_i$ can be expressed in terms of the
differences,
\begin{align*}
  u_i &= \sum_{j=1}^i (u_j - u_{j-1})\\
  &= \sum_{j=1}^i \dfrac{u_j - u_{j-1}}{h_{j-\frac{1}{2}}} h_{j-\frac{1}{2}}.
\end{align*}
Here, the convention that $u_0=0$ has been used.
From this expression, and by the Cauchy-Schwarz inequality, we derive
that 
\begin{align*}
  u_i^2 &= \left(\sum_{j=1}^i \dfrac{u_j - u_{j-1}}{h_{j-\frac{1}{2}}}
    h_{j-\frac{1}{2}}\right)^2\\
  &\leq \sum_{j=1}^i \left(\dfrac{u_j -
      u_{j-1}}{h_{j-\frac{1}{2}}}\right)^2 
  h_{j-\frac{1}{2}}\cdot \sum_{j=1}^i h_{j-\frac{1}{2}} \\
  &\leq \sum_{j=1}^{N+1} \left(\dfrac{u_j -
      u_{j-1}}{h_{j-\frac{1}{2}}}\right)^2 
  h_{j-\frac{1}{2}}\cdot 1 \\
&= |\nabla_h u_h|_0^2.
\end{align*}
Thus, for the $L^2$-norm, we have
\begin{equation*}
  |u_h|_0^2 \leq |\nabla_h u_h|_0^2 \sum_{i=1}^N h_i = |\nabla_h u_h|_0^2.  
\end{equation*}
\end{proof}

Due to the discrete Poincar\'e inequality, the bilinear  form
$a_h(\cdot,\,\cdot)$ is coercive on $V_h$. Then the existence and
uniqueness of a solution to the numerical scheme \eqref{eq:17} follows
from the Lax-Milgram theorem. An energy bound on the solution is
given by the following
\begin{lemma}
  Let $f\in L^2(0,1)$, and let $u_h$ be the corresponding discrete
  solution to the scheme \eqref{eq:16}. Then
  \begin{equation}
    \label{eq:19}
    |u_h|_{1,h} \leq |f|_0.
  \end{equation}
\end{lemma}
\begin{proof}
Replacing $v_h$ by $u_h$ in \eqref{eq:17}, and using the discrete
Poincar\'e inequality, we have
  \begin{align*}
    |u_h|_{1,h}^2 &= (\nabla_h u_h,\, \nabla_h u_h) \\
    &= (f,\,u_h)\\
    &\leq |f|_0\cdot |u_h|_0 \\
    &\leq |f|_0\cdot |u_h|_{1,h}.
  \end{align*}
The claim \eqref{eq:19} follows by dividing both sides by
$|u_h|_{1,h}$. 
\end{proof}

With the existence, uniqueness, and boundedness of the discrete
solution, we can now show that the discrete solution converges, and
the limit is a unique solution of the boundary value problem
\eqref{eq:1}--\eqref{eq:2}. 
\begin{theorem}
  Let $f\in L^2(0,1)$, and let $u_h$ be the corresponding discrete
  solution of \eqref{eq:16} (or \eqref{eq:17}). Then there exists a
  unique $u\in V$ such that, as the mesh resolution refines, 
  \begin{equation}
    \label{eq:20}
    \P_h u_h \longrightarrow \Pi u\qquad\textrm{strongly in } F,
  \end{equation}
and $u$ solves the variational problem \eqref{eq:3}. 
\end{theorem}

This result has been established in \cite{Faure:2006ky}. The
proof given below follows the same ideas of the reference cited, but is
adapted for the particular  mesh specifications used here. 

\begin{proof}
  By the uniform bound \eqref{eq:19} on $\{u_h\}_{h\in\mathcal{H}}$ and
    the weak compactness 
  of $F$, there exists a subsequence $u_{h'}$ such that
  \begin{equation}
    \label{eq:21}
    \P_h u_{h'} \equiv (u_{h'},\,\nabla_{h'}u_{h'}) \rightharpoonup
    (u,\,w) \quad \textrm{weakly in } F
  \end{equation}
By the $(C2)$ condition of Lemma \ref{lem:convergences}, $u$
belongs to $V$, and 
\begin{equation*}
  \dfrac{d u}{d x} = w.
\end{equation*}
 
Let $v\in\D(0,1)$ and $v_{h'} = \R_{h'} v$. Then by the $(C1)$
condition of Lemma
\ref{lem:convergences},
\begin{equation}
  \label{eq:22}
  \P_h\R_h v \equiv (R_h v,\,\nabla_h\R_h v) \longrightarrow \Pi v
  \equiv (v,\,v_x)\quad\textrm{ strongly in } F.
\end{equation}
Since each $u_{h'}$ is a solution to the scheme \eqref{eq:17}, and
$v_{h'}\in V_h$, we have
\begin{equation}
  \label{eq:23}
  (\nabla_{h'} u_{h'},\, \nabla_{h'}v_{h'}) = (f,\, v_{h'}).
\end{equation}
Thanks to the weak convergence of \eqref{eq:21} and the strong
convergence of \eqref{eq:22}, we can pass to the limit in
\eqref{eq:23} by sending $h'$ to zero and obtain
\begin{equation}
  \label{eq:24}
  (u_x,\, v_x) = (f,\,v),\qquad \forall v\in \D(0,1).
\end{equation}
Since $\D(0,1)$ is a dense subspace of $V$, the above relation holds
for every $v\in V$, and thus $u$ solves the variational equation
\eqref{eq:3}. 

The solution must be unique thanks to the coercivity of the bilinear
operator $a(\cdot,\,\cdot)$. 

Due to uniqueness, the whole sequence converges,
  \begin{equation}
\label{eq:25}
    \P_h u_{h} \equiv (u_{h},\,\nabla_{h}u_{h}) \rightharpoonup
    (u,\,u_x) \quad \textrm{weakly in } F
  \end{equation}

We now show that, actually, the convergence hold in the strong
topology of $F$. We consider
\begin{align*}
  |u_h - \R_h u|_{1,h}^2 &= a_h (u_h-\R_h u,\, u_h - \R_h u)\\
  &= (f, u_h) + (\nabla_h\R_h u,\, \nabla_h\R_h u) - (\nabla_h\R_h
  u,\,\nabla_h u_h) - (\nabla_h u_h,\,\nabla_h\R_h u)
\end{align*}
Thanks to the strong convergence in \eqref{eq:22} and the weak
convergence in \eqref{eq:25}, we can pass to the limit in the above,
and obtain
\begin{equation*}
  |u_h - \R_h u|_{1,h}^2 = (f,u) - (u_x, u_x) = 0.
\end{equation*}
The conclusion \eqref{eq:20} follows the observation that
\begin{align*}
  |\P_h u_h -\Pi u|_F &= |\P_h u_h - \P_h\R_h u + \P_h\R_h u - \Pi
  u|_H\\
&\leq |\P_h|\cdot |u_h - \R_h u|_{V_h} + |\P_h \R_h u - \Pi
  u|_F. 
\end{align*}
\end{proof}

\begin{remark}
  The current straightforward construction of $\R_h$ (see
  \eqref{eq:70}) may NOT be 
  extensible to high dimensional unstructured meshes, or to
  higher-order problems, e.g.~$\Delta^2$.
\end{remark}

\subsection{Error estimates}\label{sec:error-estimates}
We let $u$ be the true solution and $u_h$ be the solution of the
numerical scheme \eqref{eq:16} or \eqref{eq:17}. The
discrete $H^1$-norm of the difference (error) can be written as
\begin{align*}
  |u_h - \R_h u|_{1,h}^2 &= a_h(u_h-\R_h u,\, u_h - \R_h u)\\
  &= a_h(u_h,\, u_h - \R_h u) - a_h(\R_h u,\, u_h - \R_h u).
\end{align*}
Since $u_h$ is a solution to the discrete scheme, we can use the
relation \eqref{eq:17} to rewrite the above expression as
\begin{equation}
  \label{eq:34}
  |u_h - \R_h u|_{1,h}^2
= (f,\, u_h - \R_h u) - (\nabla_h \R_h u,\, \nabla_h(u_h - \R_h u)).
\end{equation}
By the integration by parts formula \eqref{eq:15} on the second inner
product, we obtain 
\begin{equation}
  \label{eq:26}
  |u_h - \R_h u|_{1,h}^2 = (f+\tilde\nabla_h\nabla_h \R_h u,\, u_h
  -\R_h u).
\end{equation}

On uniform meshes, we can use Taylor series expansion to show that the
truncation error $f+\tilde\nabla_h\nabla_h \R_h u$ converges to zero
as fast as $O(h^2)$, and thus $|u_h -\R_h u|_{1,h}$ converges to zeros
also at the second order. But this technique does not work for
non-uniform mesh, because the discrete Laplace operators are often not
consistent approximations to the corresponding differential
operators (\cite{Eymard:2000tt}). To see that, we expand $u(x_{i\pm
  1}$ around $u(x_i)$ by Taylor series,
\begin{align*}
  u(x_{i+1}) &= u(x_i) + h_{i+\frac{1}{2}} u_x(x_i) + \dfrac{1}{2}
  h^2_{i+\frac{1}{2}} u_{xx}(x_i) + O(h^3)\\
  u(x_{i-1}) &= u(x_i) - h_{i-\frac{1}{2}} u_x(x_i) + \dfrac{1}{2}
  h^2_{i-\frac{1}{2}} u_{xx}(x_i) + O(h^3)\\
\end{align*}
Denoting $u(x_i)$ by $u_i$, etc., we obtain from these expansions that
\begin{align*}
  \dfrac{u_{i+1} - u_i}{h_{i+\frac{1}{2}}}&= u_x(x_i) + \dfrac{1}{2}
  h_{i+\frac{1}{2}} u_{xx}(x_i) + O(h^2),\\
  \dfrac{u_{i} - u_{i-1}}{h_{i-\frac{1}{2}}}&= u_x(x_i) - \dfrac{1}{2}
  h_{i-\frac{1}{2}} u_{xx}(x_i) + O(h^2).
\end{align*}
Taking the difference between these two equations, and dividing both
sides by $h_i$, we derive
\begin{equation*}
\textcolor{rev}{[\tilde\nabla_h\nabla_h \R_h u]_i}
 \equiv \dfrac{  \dfrac{u_{i+1} - u_i}{h_{i+\frac{1}{2}}} - \dfrac{u_{i} -
    u_{i-1}}{h_{i-\frac{1}{2}}}}{h_i} = \dfrac{h_{i+\frac{1}{2}} +
  h_{i-\frac{1}{2}}}{2h_i} u_{xx}(x_i) + O(h).
\end{equation*}
The truncation error for the discrete Laplace operator can be written
as 
\begin{equation*}
  u_{xx}(x_i) - [\tilde\nabla_h\nabla_h \R_h u]_i =
  \left(1-\dfrac{h_{i+\frac{1}{2}} + h_{i-\frac{1}{2}}}{2h_i}\right)
  u_{xx}(x_i) + O(h).
\end{equation*}
Thus, on an non-uniform mesh, the discrete Laplace operator is a
consistent approximation  to the differential operator if and only if
\begin{equation*}
  \dfrac{h_{i+\frac{1}{2}} + h_{i-\frac{1}{2}}}{2h_i} = 1.
\end{equation*}
This condition is generally not true on non-uniform meshes. 

 Instead of Taylor series expansion, we integrate \eqref{eq:1} over
 $K_i$ to obtain 
\begin{equation}
  \label{eq:27}
-u_x(x_{i+\frac{1}{2}}) + u_x(x_{i-\frac{1}{2}}) = (f,\,\chi_i).  
\end{equation}
 We define the restriction operator on the dual mesh as
 \begin{equation}
   \label{eq:28}
   \tilde \R_h u_x  = \sum_{i=0}^N u_x(x_{i+\frac{1}{2}}) \chi_{i+\frac{1}{2}}.
 \end{equation}
We rewrite \eqref{eq:27} as
\begin{equation*}
  -\dfrac{u_x(x_{i+\frac{1}{2}}) - u_x(x_{i-\frac{1}{2}})}{h_i} \cdot
  h_i = (f,\,\chi_i).
\end{equation*}
Using the newly defined restriction operator $\tilde\R_h$, and the
gradient operator $\tilde\nabla_h$ on the dual mesh (see
\eqref{eq:6}), we can rewrite the above 
relation as
\begin{equation*}
  -(\tilde\nabla_h \tilde\R_h u_x,\,\chi_i) = (f,\,\chi_i).
\end{equation*}
By integration by parts, it becomes
\begin{equation}
  \label{eq:29}
  (\tilde\R_h u_x,\,\nabla_h \chi_i) = (f,\,\chi_i).
\end{equation}
We note that this relation hold for arbitrary $1\le i\le N$. 
Thus, for an arbitrary $v_h\in V_h$, the following relation holds,
\begin{equation}
  \label{eq:30}
  (\tilde\R_h u_x,\, \nabla_h v_h ) = (f,\, v_h).
\end{equation}

For arbitrary $x_{i+\frac{1}{2}}$, we note by Taylor series expansion
that, if $u\in C^2([0,1])$, 
\begin{equation}
  \label{eq:31}
  u'(x_{i+\frac{1}{2}}) = \dfrac{u(x_{i+1}) -
    u(x_i)}{h_{i+\frac{1}{2}}} + \tau_{i+\frac{1}{2}},
\end{equation}
where $\tau_{i+\frac{1}{2}}$ represents the truncation error and 
\begin{equation*}
  |\tau_{i+\frac{1}{2}}| \leq C\cdot h.
\end{equation*}
We note that $u'(x_{i+\frac{1}{2}})$ represents the flux at
$x_{i+\frac{1}{2}}$, which is the interface between cells $K_i$ and
$K_{i+\frac{1}{2}}$ (Figure \ref{fig:1d-mesh}). Relation \eqref{eq:31}
indicates the consistency 
in the approximation to the flux. Using the restriction operator
$\R_h$ and the discrete gradient operator $\nabla_h$, we can rewrite
\eqref{eq:31} as
\begin{equation}
  \label{eq:32}
  \tilde\R_h u_x = \nabla_h \R_h u + \tau_h,
\end{equation}
with
\begin{equation*}
  \tau_h = \sum_{i=0}^N \tau_{i+\frac{1}{2}}
  \chi_{i+\frac{1}{2}},\qquad |\tau_h|_\infty \leq C\cdot h.
\end{equation*}
Thus \eqref{eq:30} can be written as
\begin{equation}
  \label{eq:33}
  (\nabla_h\R_h u,\,\nabla_h v_h) = (f,\,v_h) -
  (\tau_h,\,v_h),\qquad\forall\, v_h \in V_h.
\end{equation}
Using \eqref{eq:33} in the relation \eqref{eq:34}, and by the
Poincar\'e inequality, we obtain
\begin{align*}
  |u_h - \R_h u|_{1,h}^2 &= (\tau_h,\, u_h - \R_h u)  \\
  &\leq |\tau_h|_0\cdot |u_h - r_h u|_0\\
  &\leq C|\tau_h|_0\cdot |u_h - \R_h u|_{1,h}.
\end{align*}
The estimate on the remainder term $\tau_h$ implies that
\begin{equation}
  \label{eq:35}
  |u_h - \R_h u|_{1,h} \leq C\cdot h.
\end{equation}
The $L^2$ error estimate can be trivially obtained through the
Poincar\'e inequality,
\begin{equation}
  \label{eq:36}
  |u_h - \R u|_{0,h} \leq C|u_h - \R_h u|_{1,h} \leq Ch.
\end{equation}

If $x_{i+\frac{1}{2}}$ is the midpoint of $K_{i+\frac{1}{2}}$, and
$u\in \mathcal{C}^3([0,1])$, then the
remainder $\tau_h$ is second-order in $h$, that is,
\begin{equation}
  \label{eq:37}
|\tau_h|_\infty \leq C h^2.
\end{equation}
Then the semi $H^1$ and $L^2$ error estimates are of the second order
as well,
\begin{align}
  |u_h - \R_h u|_{1,h} &\leq C\cdot h^2,\\   
 |u_h - \R_h u|_{0,h} &\leq C\cdot h^2.
\end{align}

These results are summarized as follows.
\begin{theorem}
  Let $u$ be the true solution of the elliptic boundary value problem
  \eqref{eq:1}-~\eqref{eq:2}, and $u_h$ be the discrete solution of
  the numerical scheme \eqref{eq:17}. If $u\in\C^2([0,1])$, then the
  following error estimates hold,
  \begin{align*}
    |u_h - \R_h u|_{0,h} &\leq C\cdot h,\\
    |u_h - \R_h u|_{1,h} &\leq C\cdot h.\\
  \end{align*}
If, furthermore, $x_{i+\frac{1}{2}}$ is the midpoint of
$K_{i+\frac{1}{2}}$ for each $1\leq i\leq N-1$, and $u\in\C^3([0,1])$,
  then the following error estimates hold,
  \begin{align*}
    |u_h - \R_h u|_{0,h} &\leq C\cdot h^2,\\
    |u_h - \R_h u|_{1,h} &\leq C\cdot h^2.\\
  \end{align*}
\end{theorem}

\section{The two-dimensional incompressible Stokes
  problem}\label{sec:two-dimens-incompr}
\subsection{Statement of the problem}\label{sec:statement-problem}
We let $\Omega\subset\mathbb{R}^2$ be a bounded domain with smooth
boundaries. The incompressible Stokes problem on this domain can be
stated as follows,
\begin{align}
  -\Delta\ub + \nabla p &= \bs{f}, & &\Omega\label{eq:38}\\
  \nabla\cdot\ub &= 0, &  &\Omega\label{eq:39}\\
  \ub &= 0, & &\p\Omega\label{eq:40}
\end{align}
Here $\ub = \ub(x,y)$ is a two-dimensional vector field on $\Omega$
representing the velocity, $p = p(x,y)$ a scalar field representing
the pressure, and $\bs{f} = \bs{f}(x,y)$ a two-dimensional vector
field representing the external force. The boundary condition
\eqref{eq:40} is of the Dirichlet type. Using the vector identity
\begin{equation}
  \Delta\ub = \nabla(\nabla\cdot\ub) +
  \nabla^\perp(\nabla\times\ub) \label{eq:41} 
\end{equation}
and the incompressibility condition \eqref{eq:39}, we obtain the
vorticity formulation for equation \eqref{eq:38},
\begin{equation}
  \label{eq:42}
  -\nabla^\perp (\nabla\times\ub) + \nabla p = \bs{f},\qquad \Omega.
\end{equation}
In the above, $\nabla^\perp = \kb\times\nabla$ denotes the skewed
gradient operator. Equations \eqref{eq:42}, \eqref{eq:39} and
\eqref{eq:40} form the vorticity formulation of the two-dimensional
incompressible Stokes problem. In this study, we use the vorticity
formulation, as it highlights the role of vorticity, and seems most
suitable for staggered-grid discretization techniques. 

The natural functional space for the solution of the Stokes problem is 
\begin{equation*}
  V = \left\{\ub\in H^1_0(\Omega)\,|\, \nabla\cdot \ub = 0 \textrm{ in
    } \Omega\right\}.
\end{equation*}
It is a Hilbert space under the norm
\begin{equation*}
  \| \ub \|_V = |\nabla\times\ub|_0.
\end{equation*}
By integration by parts, we obtain the weak formulation of the Stokes
problem \eqref{eq:42}, \eqref{eq:39}-~\eqref{eq:40}:
\begin{quote}
  {\itshape For each $\bs{f}\in L^2(\Omega)$, find $\ub\in V$ such that}
  \begin{equation}
    \label{eq:43}
    (\nabla\times\ub,\,\nabla\times\vb) =
    (\bs{f},\,\vb),\qquad\forall\, \vb\in V.
  \end{equation}
\end{quote}
The existence and uniqueness of a solution to this problem follows 
from the Lax-Milgram theorem. 

\subsection{Mesh specifications and an external approximation of
  $V$}\label{sec:mesh} 
\begin{figure}[h]
  \centering
  \includegraphics[width=4.5in]{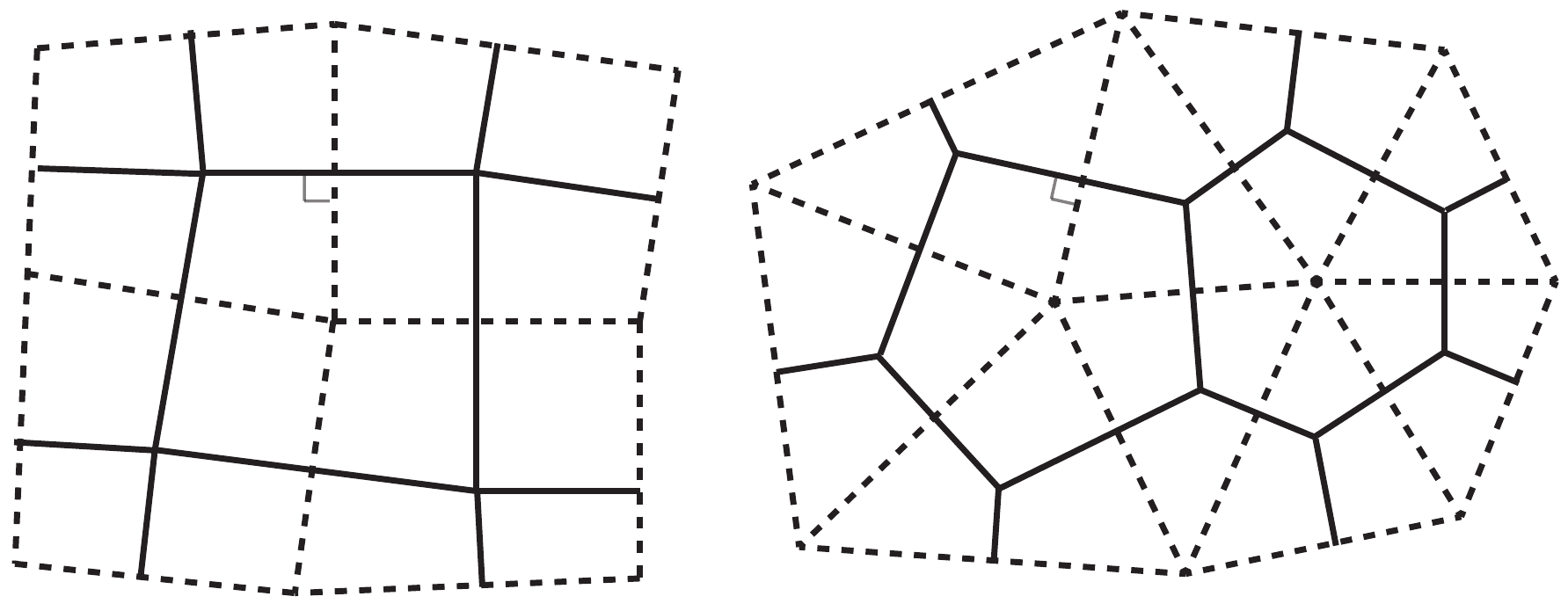}
  \caption{Generic dual meshes, with the domain boundaries passing
    through the primary cell centers. Left: a generic quadrilateral
    dual mesh. Right: a generic Delaunay-Voronoi dual mesh.}
  \label{fig:quad-dv}
\end{figure}
Our approximation of the function space is based on discrete meshes
that consist of polygons. To avoid potential technical issues with the boundary, we
shall assume that the domain $\Omega$ itself is polygonal. We make use
of a pair of staggered meshes, with one called primary and the other called
dual. The meshes consist of polygons, called cells, of arbitrary
shape, but conforming to the requirements to be specified. The centers
of the cells on the primary mesh are the vertices of the cells on the
dual mesh, and vice versa. The edges of the primary cells intersect
{\it orthogonally} with the edges of the dual cells. The line segments
of the boundary $\partial\Omega$ pass through the centers of the
primary cells that border the boundary. Thus the primary cells on the
boundary are only partially contained in the domain. 
\textcolor{rev}{Two examples of this mesh type are shown in Figure
  \ref{fig:quad-dv}.}


\begin{figure}[h]
  \centering
  \scalebox{0.65}{\includegraphics{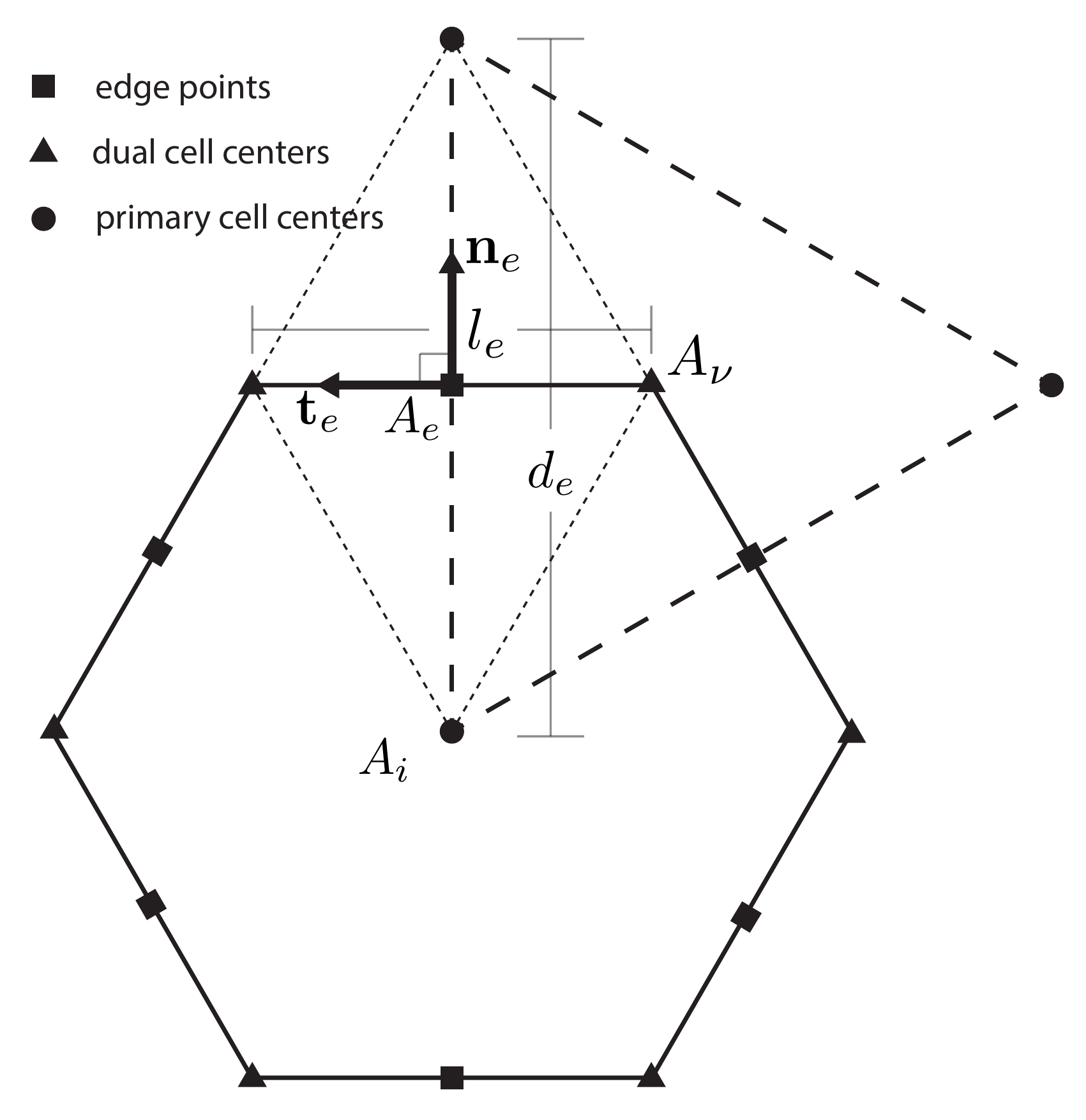}}
  \caption{Notations}
  \label{fig:notations}
\end{figure}

\begin{table}[h]
\centering
\ra{1.3}
\caption{Sets of elements defining the connectivity of a unstructured
  dual grid.}
\vspace{4mm}
\begin{tabular}{@{}ll@{}}\toprule
Set & Definition\\
\midrule
EC($i$) & Set of edges defining the boundary of primary cell $A_i$\\
VC($i$) & Set of dual cells that form the vertices primary cell $A_i$\\
CE($e$) & Set of primary cells boarding edge $e$\\
VE($e$) & Set of dual cells boarding edge $e$\\
CV($\nu$) & Set of primary cells that form vertices of dual cell $D_\nu$\\
EV($\nu$) & Set of edges that define the boundary of dual cell $D_\nu$\\
\bottomrule
\end{tabular}
\label{ta1}
\end{table}

In order to construct function spaces on this type of meshes, some
notations are in order, for which we follow the conventions made in
\cite{Ringler:2010io,Chen:2013bl}. As shown in the diagram in Figure
\ref{fig:notations}, the primary cells are denoted as $A_i,\, 1\leq
i\leq N_c + N_{cb}$, where $N_c$ denotes the number of cells that are
in the interior of the domain, and $N_{cb}$ the number of cells that
are on the boundary. We assume the cells are numbered so that $A_i$
with $1\leq i\leq N_c$ refer to interior cells.  The dual cells, which
all lie inside the domain, are denoted as $A_\nu,\,1\leq \nu\leq
N_v$. 
The area of $A_i$ (resp.~$A_\nu$) is denoted as $|A_i|$
(resp.~$|A_\nu|$). 
Each primary
cell edge corresponds to a distinct dual cell edge, and vice
versa. Thus the primary and dual cell edges share a common index $e,\,
1\leq e\leq N_e+N_{eb}$, where $N_e$ denotes the number of edge pairs
that lie entirely in the interior of the domain, and $N_{eb}$ the
number of edge pairs on the boundary, i.e., with dual cell edge
aligned with the boundary of the domain. Again, we assume that $1\le
e\le N_e$ refer to interior edges.
 Upon an edge pair $e$, the distance between the two
primary cell centers, which is also the length of the corresponding
dual cell edge, is denoted as $d_e$, while the distance between the
two dual cell centers, which is also the length of the corresponding
primary cell edge, is denoted as $l_e$. 
These two edges form the diagonals of a diamond-shaped region, whose
vertices consist of the two neighboring primary cell centers and the
two neighboring dual centers. The diamond-shaped region is also
indexed by $e$, and will be referred to as $A_e$.
The Euler formula for plannar
graphs states that the number of primary cell centers $N_c + N_{cb}$, the
number of vertices (dual cell centers) $N_v$, and the number of
primary or dual cell edges $N_e + N_{eb}$ must satisfy the relation
\begin{equation}
\label{eq:44}
  N_c + N_{cb} + N_v = N_e + N_{eb} + 1.
\end{equation}
The connectivity information of the unstructured staggered meshes is
provided by six {\it sets of elements} defined in Table \ref{ta1}. 

For each edge pair, a unit vector $\nb_e$, normal to the primary cell
edge, is specified. A second unit vector $\tb_e$ is defined as
\begin{equation}
\label{eq:45}
  \tb_e = \kb\times\nb_e,
\end{equation}
with $\kb$ standing for the upward unit vector. Thus $\tb_e$ is
orthogonal to the dual cell edge, but tangent to the primary cell
edge, and points to the vertex on the left side of $\nb_e$. For each
edge $e$ and for each $i\in \CE(e)$ (the set of cells on edge $e$, see
Table \ref{ta1}), we define the direction indicator
\begin{equation}
\label{eq:46}
  n_{e,i} = \left\{
  \begin{aligned}
    1& & &\phantom{sssssss}\textrm{if }\nb_e\textrm{ points away from primary cell
    }A_i,\\
    -1& &  &\phantom{sssssss}\textrm{if }\nb_e\textrm{ points towards primary cell
    }A_i,\\
  \end{aligned}\right.
\end{equation}
and for each $\nu\in \VE(e)$,
\begin{equation}
\label{eq:47}
  t_{e,\nu} = \left\{
  \begin{aligned}
    1& & &\phantom{sssssss}\textrm{if }\tb_e\textrm{ points away from dual cell
    }A_\nu,\\
    -1& &  &\phantom{sssssss}\textrm{if }\tb_e\textrm{ points towards dual cell
    }A_\nu.\\
  \end{aligned}\right.
\end{equation}

For this study, we make the following regularity assumptions on the
meshes. We assume that the diamond-shaped region $A_e$ is actually
convex. In other words, the intersection point of each edge pair falls
inside each of the two edges. We also assume that the meshes are
quasi-uniform, in the sense that there exists $h>0$ such that, for 
each edge $e$,
\begin{equation}
\label{eq:48}
  mh\leq l_e,\,d_e \leq Mh
\end{equation}
for some fixed constants $(m,\,M)$ that are independent of the
meshes. 
The staggered dual meshes are thus designated by $\mathcal{T}_h$.
\textcolor{rev}{For the convergence analysis, it is assumed in
  \cite{Chen:2016uw} that, for each edge pair $e$, the primary cell
  edge nearly bisect 
  the dual cell edge, and miss by at most $O(h^2)$. This assumption is
  also made here for the error analysis. Generating meshes
  conforming to this requirement on irregular domains, i.e.~domains
  with non-smooth boundaries or domains on surfaces, can be a
  challenge, and will be addressed elsewhere. But we point out that, on
  regular domains with smooth boundaries, this type of meshes can be
  generated with little extra effort in addition to the use of
  standard mesh generators, such as the 
  centroidal Voronoi tessellation algorithm (\cite{Du2003-gn, 
  Du2002-lf, Du1999-th})}.  

The approximation to the function space $V$ is given by
\begin{equation}
\label{eq:70a}
V_h = \left\{ u_h = \sum_{e = 1}^{N_e} u_e \chi_e \nb_e \,\biggr\vert\,\,
  \nabla_h\cdot u_h = 0.\right\} 
\end{equation}
Thanks to the discrete Poincar\'e inequality for vector fields (see
\cite{Chen:2016uw}, $V_h$ is a Hilbert space under the norm
\begin{equation}
\label{eq:71}
  \|u_h\|_{V_h} \equiv |u_h|_{1,h}\equiv |\tilde\nabla_h\times u_h|_0.
\end{equation}

Since all functions in $V$ are divergence free, there is a one-to-one
correspondence between $V$ and $H^2_0(\Omega)$
(\cite{Girault:1986vn}).  
We let 
$$F =L^2(\Omega)\times L^2(\Omega).$$
For each $\ub\in V$, we let $\psi\in H^2_0(\Omega)$ and $\omega\in
L^2(\Omega)$ such that 
\begin{equation}
\label{eq:72}
\ub = \nabla^\perp \psi,\qquad \nabla\times\ub = \omega.
\end{equation}
Then we redefine the isomorphism $\Pi$ from $V$ into $F$ as
\begin{equation}
\label{eq:73}
  \Pi\ub = (\psi,\,\omega)\in F,\qquad\forall \ub\in V.
\end{equation}
The space $F$ is endowed with the usual $L^2$-norm. In view of the
norm for $V$, it is clear that $\Pi$ is an isomorphism from $V$ into
$F$. It is important to note that the image of the operator $\Pi$ is
not the whole space $F$. It is a nowhere dense, closed subspace of the
latter.

Under the foregoing assumptions on the meshes, we now define the
restriction operator $\R_h$. We only need to define 
$\R_h$ for a dense subset of functions of $V$, and the definition can
then be extended to the whole space of $V$ according to a result in
\cite{Temam:1980wr}. We let 
\begin{equation}
\label{eq:74}
\mathcal{V} = \{ \ub \in \mathcal{D}(\Omega)\,|\, \nabla\cdot\ub = 0\},
\end{equation}
which is a dense subspace of $V$. For an arbitrary $\ub
\in\mathcal{V}$, there exists $\psi\in\D(\Omega)$ such that $\ub =
\nabla^\perp\psi$, and thus $\omega = \Delta\psi$. We then
define the 
associated discrete scalar field as
\begin{equation}
\label{eq:75}
\psi_h =  \sum_{\nu=1}^{N_v}\psi_\nu \chi_\nu,
\end{equation}
with
\begin{align}
  \psi_\nu &= \psi(x_\nu), & &\textrm{on interior dual cells,}\label{eq:76}\\
  \psi_\nu &= 0, & &\textrm{on dual cells on the  
    boundary,}  \label{eq:77}
\end{align} 
where $x_\nu$ is the coordinates for the center of dual cell
$\nu$. We note that the function $\psi$ has compact support on $\Omega$, and
therefore, if the grid resolution $h$ is fine enough, the
specification \eqref{eq:77} is consistent with \eqref{eq:76}. 
Finally, the
restriction operator $\R_h$ on $\ub\in\mathcal{V}$ is defined as
\begin{equation}
\label{eq:78}
  \R_h\ub = \tilde\nabla_h^\perp\psi_h.
\end{equation}
$\R_h\ub$ is divergence free in the discrete sense by construction
(see Lemma 2.6 of \cite{Chen:2016uw}). It vanishes on the boundary thanks
to the condition \eqref{eq:77} and the definition for
the skewed gradient operator on the boundary.

To define the prolongation operator $\P_h$, we note that, by the virtue
of Lemma 2.6 of \cite{Chen:2016uw}, every $u_h\in V_h$ is represented by
a scalar field $\psi\in\Psi_h$ via
\begin{equation}
\label{eq:79}
  u_h = \tilde\nabla_h^\perp \psi_h.
\end{equation}
The prolongation operator $\P_h$ is defined as
\begin{equation}
\label{eq:80}
  \P_h u_h = (\psi_h, \,\tilde\nabla_h\times u_h),\qquad\forall u_h \in V_h.
\end{equation}

The external approximation of $V$ consists of the mapping pair
$(F,\,\Pi)$ and the family of triplets
$\{V_h,\,\R_h,\,\P_h\}_{h\in\mathcal{H}}$. Concerning this approximation
we have the following claim.

\begin{theorem}\label{thm:stable-conv-h1_0}
  The external approximation that consists of the function space $F$,
  the isomorphic mapping $\Pi$, and the family of triplets $\{V_h,\,
  \R_h,\, \P_h\}_{h\in\mathcal{H}}$ is a stable and convergent
  approximation of $V$.
\end{theorem}
The proof can be found in \cite{Chen:2016uw}. 

\subsection{The MAC scheme}\label{sec:mac-scheme}
In discretizing the system, it is important to ensure that
the external forcing $\fb$ is also discretized in a consistent way.
For the sake of simplicity in the convergence proof later on, we discretize the
forcing term using its scalar stream and potential functions. For each
$\fb\in L^2(\Omega)\times L^2(\Omega)$, we let $\psi^f\in H_0^1(\Omega)$
and $\phi^f\in H^1(\Omega)/\mathbb{R}$ be such that 
\begin{equation}
\label{eq:49}
  \fb = \nabla^\perp \psi^f + \nabla \phi^f.
\end{equation}
By the famous Helmholtz decomposition theorem, the stream and
potential functions always exist and are unique, for each vector field
$\fb$ in $L^2(\Omega)\times L^2(\Omega)$ (see
\cite{Girault:1986vn}). The stream and potential functions are
discretized on the dual and primary meshes, respectively, by
averaging,
\begin{align}
  \psi^f_h &= \sum_{\nu=1}^{N_v} \psi^f_\nu\chi_\nu,& \textrm{with }
  \psi^f_\nu &= \overline{\psi^f}^{A_\nu},\label{eq:50}\\
  \phi^f_h &= \sum_{i=1}^{N_c+N_{cb}} \phi^f_i \chi_i,& \textrm{with }
  \phi^f_i &= \overline{\phi^f}^{A_i}.\label{eq:51}
\end{align}
Employing the technique of approximation by smooth functions and 
the Taylor's series expansion, we can show that the discrete scalar
fields converge to the corresponding continuous fields in the
$L^2$-norm, i.e.
\begin{align}
  \psi^f_h&\longrightarrow \psi^f & &\textrm{strongly in
  }L^2(\Omega),\label{eq:52}\\
\phi^f_h & \longrightarrow \phi^f & &\textrm{strongly in }
L^2(\Omega).\label{eq:53} 
\end{align}
With $\psi^f_h$ and $\phi^f_h$ defined as in \eqref{eq:52} and
\eqref{eq:53}, a discrete vector field can be specified, 
\begin{equation}
\label{eq:54}
   f_h = \tilde\nabla_h^\perp\psi^f_h + \nabla_h \phi^f_h.
\end{equation}
We take $f_h$ as the discretization of the continuous vector forcing
field $\fb$. We should point out that this is not the only way to
discretize the external forcing. There are ways that are probably more
convenient in real applications, e.g.~by taking $f_h$ to be the value
of the normal component of $\bs{f}$ at the intersection points of the
primary and dual cell edges. However, it will become clear later that the
expression \eqref{eq:54} is most convenient for convergence
analysis. But other discretizations consistent with this one are also
permissible. 

The discrete problem can now be stated as follows. 
\begin{quote}
 {\itshape For each $\fb\in L^2(\Omega)\times L^2(\Omega)$, let $f_h$ be defined as
   in \eqref{eq:54}. Find $u_h\in V_h$ and $p_h \in \Phi_h$ such that 
   \begin{equation}
\label{eq:55}
     -[\tilde\nabla_h^\perp\tilde\nabla_h\times u_h]_e + [\nabla_h
     p_h]_e = f_e,\qquad 1\leq e\leq N_e.
   \end{equation}
}
\end{quote}
The incompressibility condition and the homogeneous boundary
conditions on $u_h$ have been included in the specification of the
space $V_h$. It is important to note that equation \eqref{eq:55}
holds on interior edges only. On boundary edges, the computation of
$\tilde\nabla^\perp_h\tilde\nabla_h\times u_h$ will require boundary
conditions for $\tilde\nabla_h\times u_h$, which are not available a
priori.

As for the continuous problem, we multiply \eqref{eq:55} by
$v_h\in V_h$ and integrate by parts, and noticing that $v_h=0$ along
the boundary (see Lemma 2.5 of \cite{Chen:2016uw}), we obtain the
variational form of the numerical scheme,
\begin{equation}
\label{eq:56}
  (\tilde\nabla_h\times u_h,\, \tilde\nabla_h \times v_h) = 2(f_h,\,v_h).
\end{equation}
The term involving the pressure $p_h$ vanishes because of the
incompressibility condition on $v_h$. The factor 2 on the right-hand
side of \eqref{eq:56}  results from the integration-by-parts
process. It can also be directly explained by the fact that the 
inner product on the right-hand side only involves the normal
components of the vector fields. For $u_h,\,v_h\in V_h$, we define the
bilinear form
\begin{equation}
\label{eq:57}
  a_h(u_h,\,v_h) = (\tilde\nabla_h\times u_h,\, \tilde\nabla_h\times v_h).
\end{equation}
Then the variational form of the numerical scheme can be stated as
follows.
\begin{quote}
 {\itshape For each $\fb\in L^2(\Omega)\times L^2(\Omega)$, let $f_h$ be defined as
   in \eqref{eq:54}. Find $u_h\in V_h$ such that 
   \begin{equation}
\label{eq:58}
a_h(u_h,\,v_h) = 2(f_h,\,v_h),\qquad\forall v_h\in V_h.
   \end{equation}
}
\end{quote}
The existence and uniqueness of a solution $\ub$ to the problem
\eqref{eq:43} follow from the Lax-Milgram theorem; such result can be
found in \cite{Temam:2001wj}. The existence and uniqueness of a
solution $u_h$ to the discrete scheme \eqref{eq:55} or equivalently
\eqref{eq:58}, as well as its convergence to the true solution are
presented in \cite{Chen:2016uw}; see also 
\cite{Fuhrmann:2014bq}.

\subsection{Error estimates}\label{sec:error-estimates-2d}
We let $\ub$ be the true solution of \eqref{eq:38}--\eqref{eq:40}. For
the purpose of error analysis, we assume that $\ub$ is sufficiently
smooth. The exact regularity requirements will be made explicit
later. We let
\begin{equation*}
  \omega = \nabla\times\ub
\end{equation*}
be the vorticity, which is also assumed to be sufficiently
smooth. We then rewrite \eqref{eq:42} as
\begin{equation}
  \label{eq:59}
  -\nabla^\perp \omega + \nabla p = \fb.
\end{equation}
We consider an interior edge pair $\{l_e,\,d_e\}$, with $l_e$ bounded
by vertices $\nu_1$ and $\nu_2$, and $d_e$ bounded by cell centers
$i_1$ and $i_2$ (Figure \ref{fig:twocells}). The unit normal vector
$\nb_e$ on $l_e$ points from 
$i_1$ to $i_2$, and the unit tangential vector points from $\nu_1$ to
$\nu_2$. 
\begin{figure}[h]
  \centering
 \label{fig:twocells}
  \includegraphics[width=3in]{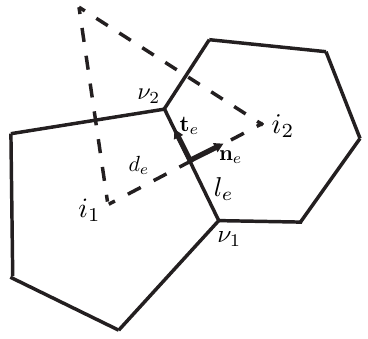}
  \caption{Edge pair}
\end{figure}
Taking the inner product of \eqref{eq:59} with $\nb_e$ and averaging
over $l_e$ yields
\begin{equation}
  \label{eq:60}
  \dfrac{\omega(x_{\nu_2}) - \omega(x_{\nu_1})}{|l_e|} +
  \dfrac{1}{|l_e|}\int_{l_e} \dfrac{\p p}{\p n_e} ds =
  \dfrac{1}{|l_e|}\int_{l_e} \fb\cdot\nb_e ds.
\end{equation}
We note that, thanks to the fact that $l_e$ and $d_e$ nearly bisect
each other, and provided that $p\in \mathcal{C}^3$, we have
\begin{equation}
  \label{eq:61}
  \dfrac{1}{|l_e|}\int_{l_e}\dfrac{\p p}{\p n_e} ds =
  \dfrac{p(x_{i_2}) - p(x_{i_1})}{|d_e|} - \tau^p_e,
\end{equation}
where 
\begin{equation*}
  \tau^p_e = O(h^2)
\end{equation*}
is the truncation error for the pressure. From \eqref{eq:49},
$\fb = \nabla^\perp\psi + \nabla \varphi$, 
\begin{equation*}
  \dfrac{1}{|l_e|}\int_{l_e}\fb\cdot\nb_e ds =
  \dfrac{1}{|l_e|}\int_{l_e} \left(-\dfrac{\p\psi}{\p t_e} + \dfrac{\p
      \varphi}{\p n_e}\right) ds.
\end{equation*}
Again, by the fact that the dual cell edges $l_e$ and $d_e$ nearly
bisect each other, and provided that $\psi^f$ and $\varphi^f$ are
sufficiently smooth,  the averaged quantities on the right-hand side can
be approximate by the corresponding difference formulae to the second
order, and therefore we have
\begin{equation}
  \dfrac{1}{|l_e|}\int_{l_e}\fb\cdot\nb_e ds =
  -\dfrac{\psi^f(x_{\nu_2}) - \psi^f(x_{\nu_1})}{|l_e|} +
  \dfrac{\varphi^f(x_{c_2}) - \varphi^f(x_{c_1})}{|d_e|} +
  \tau^f_e,\label{eq:62}
\end{equation}
where $\tau^f_e$ is the local truncation error the $\fb$, and
\begin{equation*}
  \tau^f_e = O(h^2).
\end{equation*}
We note that the external forcing $\fb$ is discretized as
\begin{equation*}
  f_h = \tilde\nabla^\perp_h \psi_h + \nabla_h \varphi_h.
\end{equation*}
Therefore, we have
\begin{equation}
  \dfrac{1}{|l_e|}\int_{l_e}\fb\cdot\nb_e ds = 
f_e + \tau^f_e.\label{eq:63}
\end{equation}
Combining \eqref{eq:60}, \eqref{eq:61}, and \eqref{eq:63}, we obtain
\begin{equation}
\label{eq:64}
  \dfrac{\omega(x_{\nu_2}) - \omega(x_{\nu_1})}{|l_e|} +
  \dfrac{p(x_{c_2}) - p(x_{c_1})}{|d_e|} =
  f_e + \tau^p_e + \tau^f_e.
\end{equation}
Using the gradient and the skewed gradient operators, the above
equation can be written as 
\begin{equation}
  \label{eq:65}
  -\left[ \tilde\nabla^\perp_h \R_h\omega\right]_e + \left[\nabla_h
    \R_h p\right]_e = \left[ f_h\right]_e + \tau^p_e + \tau^f_e,\qquad
  1\le e \le N_e.
\end{equation}
We let $v_h \in V_h$ be arbitrary. Multiplying \eqref{eq:65} by $v_h$
and integrating by parts yields the identity
\begin{equation}
  \label{eq:66}
  \left( \R_h \omega,\,\tilde\nabla_h\times v_h\right) = 2(f_h,\, v_h)
  + 2(\tau^p_h + \tau^f_h,\, v_h).
\end{equation}
It has been shown in \cite{Chen:2016uw} that the discrete curl of
$\R_h u$ approximate $\R_h\omega$ with a first-order error, provided
that $\psi^u\in \C^3(\Omega)$,
\begin{equation}
  \label{eq:67}
  \tilde\nabla_h \times \R_h u = \R_h \omega + \tau^\omega_h,
\end{equation}
where 
\begin{equation*}
  \tau^\omega_h = O(h).
\end{equation*}
This shows that the discrete curl operator $\tilde\nabla_h\times (
\cdot)$ is consistent with the curl operator
$\nabla\times(\cdot)$. But, since the truncation error $\tau^\omega_h$
is only of the first order, the approximation to the Laplace operator,
which is one order higher than the curl operator, will not be
consistent, in general, on non-uniform meshes. 

Substituting \eqref{eq:67} into \eqref{eq:66} yields
\begin{equation}
\label{eq:68}
  \left( \tilde\nabla_h\times\R_h u,\,\tilde\nabla_h\times v_h\right) = 2(f_h,\, v_h)
  + 2(\tau^p_h + \tau^f_h,\, v_h) +
  (\tau^\omega_h,\,\tilde\nabla\times v_h).
\end{equation}
Thus the truncation error consists of three parts: the pressure
gradient discretization error $\tau^p_h$, the external forcing
discretization error $\tau^f_h$, and the vorticity discretization
error $\tau^\omega_h$.

Using the relation \eqref{eq:58}, the error in the discrete solution
$u_h$ in the semi-$H^1$ norm can be calculated as
\begin{align*}
  |u_h - \R_h u|_1^2 =& a_h (u_h - \R_h u,\, u_h - \R_h u)\\
  =& a_h (u_h,\, u_h - \R_h u) - a_h (\R_h u,\, u_h - \R_h u)\\
  =& 2(f_h,\, u_h-\R_h u) - \left(\tilde\nabla_h \times \R_h u,\,
  \tilde\nabla_h\times(u_h - \R_h u)\right).
\end{align*}
Using the relation \eqref{eq:68}, and the discrete Poincar\'e
inequality, we obtain
\begin{align*}
  |u_h - \R_h u|_1^2 &=  2(\tau^p_h + \tau^f_h,\, v_h) +
  (\tau^\omega_h,\,\tilde\nabla\times v_h)\\
  &\leq 2|\tau^p_h|_0\cdot |u_h - \R_h u|_0 + 2|\tau^f_h|_0\cdot|u_h -
  \R_h u|_0 + |tau^\omega_h|_0\cdot |u_h - \R_h u|_1\\
   &\leq C\left(|\tau^p_h|_0 + |\tau^f_h|_0 + |\tau^\omega_h|_0\right)\cdot
     |u_h - \R_h u|_1.
\end{align*}
The estimate for the error in the $H^1$-norm thus results,
\begin{equation}
  \label{eq:69}
  |u_h - \R_h u|_1 \leq C\left(|\tau^p_h|_0 + |\tau^f_h|_0 +
    |\tau^\omega_h|_0\right) \leq C h.
\end{equation}
In the above, $C$ is a constant independent of the grid resolution $h$
or the solution $u$.
By another trivial application of the Poincar\'e inequality, the
estimate for the $L^2$-norm of the error is obtained, which is of the
same order. 

The above result is formally summarized in
\begin{theorem}
  If $\psi^u\in C^3(\Omega)$, $\psi^f\in C^3(\Omega)$, $\varphi^f\in
  C^3(\Omega)$, and $p \in 
  C^3(\Omega)$, then the discrete solution $u_h$ is first-order
  accurate in the discrete $H^1$-norm, i.e.
  \begin{equation*}
    |u_h - \R_h u|_1 = O(h).
  \end{equation*}
\end{theorem}

It is worth noting that the first-order error in the discrete solution
$u_h$ comes from the first-order error in the discretization of the
vorticity, which is the Laplace of the streamfunction. On certain
special meshes, such as the rectangular mesh, or the equilateral
triangular mesh, it is well-known that the discretization error for
the Laplace operator is second-order accurate. In these cases, the
$H^1$-norm of the error in $u_h$ will also be of second
order. However, at this point, we do not have a general
characterization of meshes on which such result holds. 

\section{Concluding remarks}
Analysis of finite difference finite volume methods is hard and often
{\itshape ad-hoc}, due to the lack of a variational framework;
\textcolor{rev}{progresses have been made in this area by various authors
(\cite{Faure:2006ky,Gie:2010vy,  Droniou2013-mq, Gie:2015tw}).} Error 
analysis of FDFV methods is further hindered by the fact that FDFV
schemes on non-uniform meshes are often inconsistent with the
underlying differential equations. The present work adopts a new
approach that combines the consistency in lower-order differential
operators with the external approximation framework for
function spaces (\cite{Cea:1964vy,{Temam:2001wj}}). The new combined
approach is first applied to the one-dimensional elliptic problem on a
non-uniform mesh. A first-order convergence rate, in both $L^2$ and
$H^1$ norms, is obtained, which agrees with the results obtained by a
different approach in \cite{Eymard:2000tt}. The new combined approach
is also applied to the MAC scheme for the incompressible Stokes
problem on an unstructured mesh. A first-order convergence
rate, in both $L^2$ and $H^1$ norms, is obtained. The results
presented here improve those of Nicolaides (\cite{Nicolaides:1992vs})
in that they hold on unstructured meshes. The current analysis also
shows that, on certain special meshes, including the rectangular mesh
considered by Nicolaides, the convergence rate is actually of the second
order. 

The new approach presented here has only a few general contingencies
such as a stable and convergent external approximation of the function
space associated with the continuous problem, and consistency in the
approximation of the lower-order differential operators that appear in
the equations. Thus, it is expected that the approach
can be applied to derive error estimates for a wide range of
problems, such as the steady-state Navier-Stokes problem, the
time-dependent Navier-Stokes problem, the quasi-geostrophic equation
for large-scale geophysical flows, and the shallow water equations. The
present paper demonstrates the
effectiveness of the new approach. Endeavors on the aforementioned
systems is currently on-going, and the results will be presented
elsewhere.  

\bibliographystyle{amsplain}
\bibliography{references}

\end{document}